\documentclass[11pt,twoside]{amsart}
\usepackage{amsmath, amsthm, amscd, amsfonts, amssymb, graphicx, color}
\usepackage[bookmarksnumbered, plainpages]{hyperref}
\newtheorem{thm}{Theorem}[section]

\newtheorem{conj}[thm]{Conjecture}

\newtheorem{rem}[thm]{\bf{Remark}}

\numberwithin{equation}{section}
\def\pn{\par\noindent}

\begin{document}


\title{NONINNER AUTOMORPHISMS OF ORDER p FOR FINITE p-GROUPS OF COCLASS 2}
\author{A. Abdollahi$^*$, S. M. Ghoraishi, Y. Guerboussa,  M. Reguiat and B. Wilkens}

\thanks{{\scriptsize
\hskip -0.4 true cm MSC(2010): Primary: 20D45 Secondary: 20E36.
$*$Corresponding author}}
\maketitle


\begin{abstract}
Every  finite $p$-group of coclass $2$ has a noninner automorphism of order $p$ leaving  the center elementwise fixed.
\end{abstract}



%


\section{\bf Introduction}
\vskip 0.4 true cm

Let $G$ be a finite nonabelian $p$-group. A longstanding conjecture asserts that every finite nonabelian $p$-group admits a noninner
automorphism of order $p$ (see also Problem 4.13 of [Ko]). This conjecture has been settled  for various  classes of $p$-groups $G$ as follows:
\begin{itemize}
\item if $G$ is regular \cite{S, DS};
\item if $G$ is nilpotent of class $2$ or $3$ \cite{A,AGW,L};
\item if the commutator subgroup of $G$ is cyclic \cite{jam};
\item if $G/Z(G)$ is powerful, \cite{AB};
\item if $C_G(Z(\Phi(G)))\neq \Phi(G)$ \cite{DS};
\end{itemize}
There are some other partial results on the conjecture \cite{SMG,shab}.
In most of the above cited results on the conjecture, it is proved that $G$ has often a noninner automorphism of order $p$ leaving the center $Z(G)$  or Frattini subgroup $\Phi(G)$ of $G$ elementwise fixed.

It is proved in \cite{AG} that if  $G$ is regular or  nilpotent of class $2$ or if  $G/Z(G)$ is powerful, then $G$ has a noninner automorphism of order $p$ acting trivially on $Z(G)$. The following conjecture was put forwarded in \cite{AG}:

\begin{conj}[Conjecture 2 of \cite{AG}] \label{conj2} Every finite nonabelian $p$-group admits a noninner automorphism of order $p$ leaving
the center elementwise fixed.
\end{conj}

Our main result   is to confirm the validity of  Conjecture \ref{conj2} for $p$-groups of coclass $2$.
\begin{thm}\label{main}
Every  finite $p$-group of coclass $2$ has a noninner automorphism of order $p$ leaving  the center elementwise fixed.
\end{thm}
\section{\bf Preliminaries}
\vskip 0.4 true cm

 Let $G$ be a nonabelian finite $p$-group. We use the following facts in the proof.

\begin{rem}[{\cite[Theorem]{DS}}]\label{DS}
    If $C_G(Z(\Phi(G)))\neq \Phi(G)$, then $G$  has a noninner automorphism of order $p$ leaving   the Frattini subgroup $\Phi(G)$ elementwise fixed.
\end{rem}
\begin{rem}[{\cite[Lemma 2.2]{AB}}]\label{AB-lem}
If $d(Z_2(G)/Z(G))\neq d(G)d(Z(G))$ then $G$ has a  noninner automorphism of order $p$ leaving the Frattini  subgroup $\Phi(G)$ of $G$ elementwise fixed.
\end{rem}
\begin{rem}\label{jam}
 If the commutator subgroup of $G$ is cyclic, then  $G$ has a noninner automorphism of order $p$ leaving $\Phi(G)$ elementwise fixed whenever
$p > 2$, and leaving either $\Phi(G)$ or $Z(G)$ elementwise fixed whenever $p = 2$.
\end{rem}
\begin{rem}[{\cite[Corollary 1.7 and Theorem 1.2]{berk}}]\label{berk}
Let $G$ be a  nonabelian $2$-group of coclass $1$. Then $G$ it is isomorphic to one of the following groups:
\begin{itemize}
\item
$D_{2^{n+1}}=\langle a, b\mid a^{2^n}=b^2=1,  bab=a^{-1}\rangle$ the dihedral group of order $2^{n+1}$ for some $n\geq 2$.
\item
$Q_{2^{n+1}}=\langle a,  b \mid  a^{2^k}=1, a^{2^{n-1}}=b^2, b^{-1}ab=a^{-1}\rangle$ the generalized
quaternion group of order $2^{n+1}$ for some $n\geq 2$.
\item
$SD_{2^{n+1}}=\langle a, b\mid a^{2^n}=b^2=1,  bab=a^{-1+2^{n-1}}\rangle$ the semidihedral group of order $2^{n+1}$ for some $n\geq 3$.
\end{itemize}
Note that  the  nilpotency class of $G$ is $n$.
\end{rem}
Recall that a group $G$ is called capable if it is the group of inner automorphisms
of some group, that is, if there exists a group $H$ with $H/Z(H)\cong G$.
\begin{rem}[{\cite[Theorem]{Sh}}]\label{capable}
 The generalized quaternion group  of order at leat $8$ and  the semidihedral
group of order at least $16$, cannot be normal subgroups of a capable group.
\end{rem}
\section{\bf Proof of Theorem \ref{main}}
\vskip 0.4 true cm

We first  prove the following
\begin{thm}\label{main0}
   Let $G$ be a  finite $p$-group of coclass $2$. Then $G$ admits a noninner automorphism of order $p$.
\end{thm}
\begin{proof}
Suppose, to the contrary,
  that every automorphism of order $p$ of  $G$ is inner, By Remark \ref{AB-lem} we may assume that $d(Z_2(G)/Z(G))=d(G)d(Z(G))$. Since $G$ is of  coclass $2$,  $|Z(G)|=p$ and
$d(G)=2$, $Z_2(G)/Z(G)$ is elementary abelian of rank $2$.
This implies that
\begin{itemize}
  \item
  $Z_2(G)$ is a noncyclic abelian group of order $p^3$. Indeed,  $g^p\in Z(G)$ for every $g\in Z_2(G)$.  Thus  $[Z_2(G),\Phi(G)]=1.$
  Now it follows from Remark \ref{DS} that $Z_2(G)\leq Z(\Phi(G))$.  And,
  \item
  $G/Z_2(G)$ is of coclass $1$.
\end{itemize}

Let $\tilde G=G/Z_2(G)$ and $n=cl(\tilde G)$.
\paragraph{\bf Case $p\geq 3$}
By \cite{AB},  we may assume that $G$ is not powerful, that is
\begin{align}\label{not powerful}
  \gamma_2(G)\not\leq G^p
\end{align}

Let  $u\in \Omega_1(Z_2(G))\setminus Z(G)$ and $M=C_G(u)$.  Then $M$  is a
maximal subgroup of $G$. Let  $y\in M\setminus \Phi(G)$,  $x\in G\setminus M$. Assume  $L=\gamma_3(G)G^p$,
$w=[x,y]$ and $z=[x,u]$. Thus $Z(G)=\langle z\rangle$.

The map  $\alpha:x\mapsto xu$ , $m\mapsto m$, for all
$m\in M$ extends to an automorphism of order
$p$ of $G$. By assumption  $\alpha=\theta_t$, the inner automorphism induced by some $t\in G$.
Since $g^{-1}\alpha(g)\in Z_2(G)$, for all $g\in G$, it follows that $t\in Z_3(G)$ and since $\alpha(x)=xu$, $t\notin Z_2(G)$.
Since $t\in C_G(M)\leq C_G(\Phi(G))\leq Z(\Phi(G))$ and $G/Z_2(G)$ has coclass 1, we obtain  $Z_3(G) = \langle t,\, Z_2(G) \rangle$  and so
 \begin{align}\label{Z3}
 Z_3(G)\leq Z(\Phi(G)).
 \end{align}
We have seen that $\langle t \rangle Z_2(G) = \gamma_{n}(G) Z_2(G),$ whence $u = [x,\,t]$ lies in $\gamma_{n+1}(G) \langle z \rangle = \gamma_{n+1}(G) $.   It follows from \eqref{Z3} that  $n \geq 2,$ whence $\langle u,\, z \rangle \leq L.$\\

 Let $\overline{G} = G/L.$ Then $\overline{G}$ is a two-generator group of exponent $p$ and class at most $2,$ which makes it either elementary abelian of order $p^2$ or extraspecial of order $p^3$ and exponent $p.$ If $\overline{G}$ is abelian, then $\gamma_2(G) \leq L = \gamma_3(G)G^p $, which is possible only if
 $\gamma_2(G)\leq G^p$  contradicting \eqref{not powerful}. Thus $\overline{G}$ is extraspecial.\\
  Every element of $\bar{v}$ of $\overline{G}$ is uniquely represented as a product $\bar{v} = \bar{x}^i \bar{y}^j \bar{w}^k$($i,\, j,\, k \in \{0, \ldots, p-1\}$).  Let
define $\bar{f}: \, \bar{G} \rightarrow \langle u,\, z \rangle$ by $\bar{f} (\bar{x}^i \bar{y}^j \bar{w}^k) = u^j z^k.$ As $[L,\, Z_2(G)] = 1,$  the group $V = \langle u,\, z \rangle$ becomes a $\overline{G}$-module via $v^{\bar{g}} = v^g$ ($g \in G,\, v \in V$).\\
\noindent We claim that $\bar{f}$ is a derivation: Indeed, for $j,\, \ell \in \mathbb{Z},$ $[\bar{y}^j,\,\bar{x}^{\ell}]= \bar{w}^{-j\ell}$ and
 $[u^j,\,x^{\ell}] = z^{-j\ell},$ whence, for $i,\, j,\, k,\, i',\, j',\, k' \in \mathbb{Z},$ we have
 \begin{align*}
\bar{f} (\bar{x}^i \bar{y}^j \bar{w}^k \bar{x}^{i'} \bar{y}^{j'} \bar{w}^{k'})&=
\bar{f} (\bar{x}^{i+i'} \bar{y}^{j+j'} \bar{w}^{k+k' - ji'})\\
&= u^{j+j'} z^{k+k' - ji'}\\
&=u^{j+j'}[u^j,\, x^{i'}]  z^{k+k'} \\
&= (u^j z^k)^{x^{i'}y^{j'} w^{k'}} u^{j'} z^{k'}\\
& = \bar{f}(\bar{x}^i \bar{y}^j \bar{w}^k)^{\bar{x}^{i'} \bar{y}^{j'} \bar{w}^{k'}} \bar{f}(\bar{x}^{i'} \bar{y}^{j'} \bar{w}^{k'}).
 \end{align*}
\noindent Now define $f: \, G \rightarrow V$ by $f(g) = \bar{f}(\bar{g})$ ($g \in G$). It is clear that $f$ is a derivation. Let  $\sigma_f \in Aut(G)$ be defined by $\sigma_f(g) = g f(g)$ ($g \in G$). Observe that
$[G,\,\sigma_f] = V \leq L \leq C_G(\sigma_f).$ As $V$ is elementary abelian, it follows that
  $\sigma^p_f  (g) = g f(g)^p = g$ ($g \in G$)i.e. $o(\sigma_f) = p.$\\
If $\sigma_f=\theta_g$ is an inner automorphism  induced by the element $g$, then $[G,g]=[G,\sigma_f]\leq Z_2(G)$. Thus $g\in Z_3(G)$. By \eqref{Z3}, $g\in Z(\Phi(G))$ and so $w^g=w$. Now $w^g=\sigma_f(w)=wz$  and so $z=1$ a contradiction. This completes the proof in the case $p\geq 3$.\\

\paragraph{\bf Case $p=2$}
As we assumed, $cl(\tilde G)=n$. Thus $cl(G)=n+2$, and  $\gamma_i(G)Z_2(G)=Z_{n+3-i}(G)$ for all $i\in\{1,\dots,n\}$.  Furthermore by \cite{A} and \cite{AGW} we may assume $n \geq 2$.  Now it follows from  Remark  \ref{berk}, that $\widetilde G$ is isomorphic to either $D_{2^{n+1}}$, $Q_{2^{n+1}}$ or $SD_{2^{n+1}}$.

Since  $\widetilde G$ is capable,  Remark \ref{capable} implies that
$G$ is isomorphic to $D_{2^{n+1}}$.
     It follows that $G=\langle x,y\rangle$ for some $x,y\in G$ such that $$\widetilde G=\langle \tilde x,\tilde y\mid \tilde y^{2^n}=\tilde 1,\, \tilde x^2=\tilde 1,\, \tilde y^{\tilde x}={\tilde y}^{-1}\rangle,$$ where ${\tilde x}=xZ_2(G)$ and ${\tilde y}=yZ_2(G)$. We obtain
\begin{itemize}
  \item
  $[x,y]=y^{2}t$ for some $t\in Z_2(G)$,
  \item
   $[x,y^2]\in y^4Z(G)$,
  \item
  $[x,y^4]=y^8$ and so $\langle y^4\rangle\lhd G$.
  \item
  $\gamma_i(G)\leq \langle y^{2^{i-1}}\rangle Z(G)\leq \langle y^{2^{i-1}}\rangle$ for $3\leq i\leq n+3$.
 \end{itemize}
Now since $\gamma_{n+2}(G)\leq Z(G)$, it follows that $Z(G)=\langle y^{2^{n+1}}\rangle$.
Thus
$Z_2(G)=\langle y^{2^{n}}\rangle\times\langle v\rangle$ for some element $v$ of order $2$.

 Thus $t=y^{2^ni}v^j$ for some integers $i,j$. Then  $y^2t=y^{2(1+2^{n-1}i)}v^j$ and  $(y^2t)^2=y^{4(1+2^{n-1}i)}$. Therefore $G'=\langle [x,y], \gamma_3(G)\rangle\leq \langle y^2t, y^4\rangle=\langle y^2t\rangle$ is cyclic. Now it follows from Remark \ref{jam} that   $G$ has a noninner automorphism of order $2$  leaving either $\Phi(G)$ or $Z(G)$ elementwise fixed, a contradiction.
    This completes the proof.

\end{proof}
\noindent{\bf Proof of Theorem \ref{main}.}\\
Let $G$ be a finite $p$-group of coclass $2$. Then $|Z(G)|\in\{p,p^2\}$.

If $|Z(G)|=p$ then every $p$-automorphism of $G$ leaving $Z(G)$ elementwise fixed. Thus Theorem \ref{main0} completes the proof. Hence we may assume that $|Z(G)|=p^2$.

If $Z(G)\not\leq \Phi(G)$, then for some $g\in Z(G)$ and  some maximal subgroup $M$ of $G$ we have $g\notin M$. Thus  $cl(M)=cl(G)$ and so  $M$ is of  maximal class. It follows from \cite[Corollary 2.4.]{AB} that  $M$ has a noninner automorphism $\alpha$ of order $p$ leaving $\Phi(M)$ elementwise fixed. Now the map given by $g\mapsto g$ and $m\mapsto m^\alpha$ for all $m\in M$,
 determines a noninner automorphism of order $p$ leaving  $Z(G)$ elementwise fixed.
 Therefore we may assume that $Z(G)\leq \Phi(G)$. It follows from Remark \ref{DS} that $G$ has a noninner automorphism $\alpha$ of order $p$ leaving $\Phi(G)$ elementwise fixed. Since $Z(G)\leq \Phi(G)$, $\alpha$ leaves $Z(G)$ elementwise fixed. This completes the proof. $\hfill \Box$\\



\bigskip
\bigskip

{\footnotesize \pn{\bf Alireza Abdollahi}\; \\ Department of Mathematics, University of Isfahan, Isfahan 81746-73441, Iran\\
{\tt Email: a.abdollahi@math.ui.ac.ir}\\

{\footnotesize \pn{\bf S. M. Ghoraishi}\; \\ The Isfahan Branch of the School of Mathematics, Institute for Research in
Fundamental Sciences (IPM) Isfahan, Iran \\
{\tt Email: ghoraishi@ipm.ir}\\

{\footnotesize \pn{\bf Yassine Guerboussa}\; \\ University Kasdi Merbah Ouargla, Ouargla, Algeria \\
{\tt Email: yassine\_guer@hotmail.fr}\\

{\footnotesize \pn{\bf Miloud Reguiat}\; \\ University Kasdi Merbah Ouargla, Ouargla, Algeria \\
{\tt Email: reguiat@gmail.com}\\

{\footnotesize \pn{\bf B. Wilkens}\; \\ Department of Mathematics, University of Botswana, Private Bag 00704, Gaborone, Botswana \\
{\tt Email: wilkensb@mopipi.ub.bw}\\

\end{document}